\documentclass[12pt]{article}
\usepackage{amsmath}
\usepackage{amsthm}
\usepackage{amsfonts}
\usepackage{amssymb}
\usepackage{mathtools}
\usepackage{hyperref}

\newtheorem{thm}{Theorem}[section]
\newtheorem{lem}[thm]{Lemma}

\newtheorem{cor}[thm]{Corollary}

\theoremstyle{definition}

\newcommand{\ceil}[1]{\left\lceil#1\right\rceil}

\let\al=\alpha
\let\dt=\delta

\let\ep=\epsilon
\let\vph=\varphi
\let\abs=\envert
\let\sq=\sqrt
\let\wt=\widetilde

\theoremstyle{remark}

\begin{document}
\title{Explicit Chen's theorem
}
\author{Tomohiro Yamada}
\date{}
\maketitle

\begin{abstract}
We show that every even number $>\exp\exp 36$ can be represented
as the sum of a prime and a product of at most two primes.
\end{abstract}

\section{Introduction}\label{intro}

In a letter of 1742 to Euler, Goldbach conjectured that every integer greater then $2$
is the sum of three primes {\it including $1$}, which is equivalent that
every even integer $N\geq 4$ is the sum of two primes (not including $1$) or of the form $p+3$ with $p$ prime.

Euler replied that this is equivalent to the statement
that every even integer $N\geq 4$ is the sum of two primes.

An weaker conjecture is that every odd integer $N\geq 7$ can be represented as the sum of three primes.
Vinogradov\cite{Vin1}\cite{Vin2}\cite[Chapter 8]{Nat} showed that every sufficiently large odd integer can be represented as the sum of three primes.
His student K. Borozdin\cite{Bor} proved that $3^{3^{15}}$ is large enough.
Chen and Wang\cite{CW1} reduced the constant to $\exp\exp 11.503$,
Chen and Wang\cite{CW2} to $\exp\exp 9.715$ and Liu and Wang\cite{LW} to $\exp 3100$.
Deshouillers, Effinger, te Riele and Zinoviev\cite{DERZ} showed that
the Generalized Riemann Hypothesis gives the weaker conjecture.
Recently, Harald Helfgott claimed to have proved three prime conjecture unconditionally.

Contrastly, the ordinary Goldbach's conjecture is still unsolved.
A well-known partial result is the theorem of Chen\cite{Che1}\cite{Che2},
who proved that every sufficiently large even number can be represented as the sum of a prime
and the product of at most two primes.  Ross\cite{Ros} gave a simpler proof.
Nathanson\cite[Chapter 10]{Nat} gave another proof based on
Iwaniec's unpublished lecture note.  However, they did not give an explicit constant
above which every even number can be represented as $p+P_2$.
The purpose of this paper is to give an explicit constant for Chen's theorem;
every even number $>\exp\exp 36$ can be represented as the sum of a prime and
a product of at most two primes.
Indeed, we shall prove the following result:

\begin{thm}\label{th1}
Let $\pi_2(N)$ denote the number of representations of a given integer $N$ as the sum of a prime
and a product of at most two primes.
If $N$ is an even integer $>\exp\exp 36$, then we have
\begin{equation}
\pi_2(N)>\frac{0.007U_N N}{\log^2 N},
\end{equation}
where
\begin{equation}
U_N=2e^{-\gamma}\prod_{p>2}\left(1-\frac{1}{(p-1)^2}\right)\prod_{p>2, p\mid N}\frac{p-1}{p-2}.
\end{equation}
\end{thm}

Our argument is based on Nathanson's one, which used
Rosser-Iwaniec linear sieve to give upper and lower bounds for numbers of sifted primes,
combining explicit error terms for the disribution of primes in arithmetic progressions
and explicit Rosser-Iwaniec linear sieve, which are given in other papers
by the author \cite{Ymd1}\cite{Ymd2}.

However, possible existence of a Siegel zero prevents from making the size of error term
in Rosser-Iwaniec linear sieve explicit.
There are two cases --- the exceptional modulus is large or small.
If the exceptional modulus is small, then we can see that the contribution of
the Siegel zero can be absorbed into error estimates concerning the distribution of
primes in arithmetic progression (see Lemma \ref{lm24}).
In the other case, when the exceptional modulus is large, it is easy
to avoid a possible Siegel zero in the argument to estimate upper bounds
since we can exclude a prime dividing the exceptional modulus from sifting primes.
However, we cannot directly avoid a possible Siegel zero in the argument
to estimate lower bounds.  In order to overcome this obstacle, we use a variant of
inclusion-exclusion principle and both upper bound and lower bound sieves,
as performed in Section \ref{LS}.  Thus we can obtain explicit bounds.

So that, our argument can be divided into four parts:
error estimates involving the number of primes in arithmetic progressions based on estimates in \cite{Ymd1},
explicit error terms in Rosser-Iwaniec linear sieve shown in \cite{Ymd2},
upper bounds and lower bounds for various sets of sifted primes,
and the final conclusion.

For calculations of constants, we used PARI-GP.  Our script is available from
\url{http://tyamada1093.web.fc2.com/math/files/prim0009pari.txt}

\section{Preliminary results}\label{SecPre}

In this section, we shall introduce some preliminary results, involving explicit estimates for
various quantities involving the number of primes in arithmetic progressions.

We begin by noting that, in this paper, $\theta$ denotes a quantity with $\abs{\theta}\leq 1$
taking different values at each occurence.  It can easily be distinguished
from Chebyshev functions.

We shall introduce a partial-sum type inequality.

\begin{lem}\label{lm21}
Let $f(x)$ be a monotone function defined in $w\leq x\leq z$,
$c(n)$ be an arithmetic function satisfying
\begin{equation}
\sum_{x\leq n<y}c(n)\leq g(y)-g(x)+E
\end{equation}
for some constant $E$ whenever $w\leq x\leq y<z$.
Then we have
\begin{equation}
\sum_{w\leq n<z}c(n)f(n)\leq\int_w^z f(t)g^\prime(t)dt+E\max\{f(w), f(z)\}.
\end{equation}
\end{lem}
\begin{proof}
This is Lemma 1, (ii) in \cite[p.p. 30--31]{Gre}.
\end{proof}

We shall often use the following explicit estimates.

\begin{lem}\label{lm22}
For $n\geq 3$,
\begin{equation}\label{eq21}
\omega(n)<\frac{1.3841\log n}{\log\log n}.
\end{equation}
For $x\geq 2973$,
\begin{equation}\label{eq22}
\prod_{p\leq x}\left(1-\frac{1}{p}\right)=\frac{e^{-\gamma}}{\log x}\left(1+\frac{\theta}{5\log^2 x}\right).
\end{equation}
Moreover, we have, for any real numbers $a>1$ and $b>10372$,
\begin{equation}\label{eq23}
\sum_{a\leq p<b}\frac{1}{p}<\log\log b-\log\log a+\frac{1}{5\log^2 a}+\frac{8}{15\log^3 a}
\end{equation}
\end{lem}
\begin{proof}
(\ref{eq21}) is Theorem 11 of \cite{Rob}.
(\ref{eq22}) and (\ref{eq23}) follow from Theorem 6.12 and Theorem 6.10 in \cite{Dus} respectively.
\end{proof}

Henthforth we shall give explicit estimates for various quantities involving
the error terms concerning to the number of primes in arithmetic progressions.
Let $E_f(x; k, l)$ denote the error function $f(x; k, l)-\frac{f(x)}{\vph(k)}$
for $f=\pi$ (i.e. $f(x)=\pi(x)$ and $f(x; k, l)=\pi(x; k, l)$), $\theta$ or $\psi$.

\begin{lem}\label{lm23}
Let $x>X_1=\exp\exp 11.7$ and $k<\log^{10} x$ be an integer.
Let $E_0=1$ and $\beta_0$ denote the Siegel zero modulo $k$ if it exists and $E_0=0$ otherwise.
Then we have
\begin{equation}
\quad \frac{\vph(q)}{x}\abs{E_\psi(x; k, l)}<\frac{0.000012}{\log^8 x}+E_0\frac{x^{\beta_0-1}}{\beta_0}.
\end{equation}
\end{lem}
\begin{proof}
This is Theorem 1.1 of \cite{Ymd1} with $(\al_1, \al, Y_0)=(10, 8, 11.7)$.
\end{proof}

Define $\Pi(s, q)=\prod_{\chi\pmod{q}}L(s, \chi)$ and let $R_0=6.397$ and $R_1=2.0452\cdots$.
Theorem 1.1 of Kadiri\cite{Kad} states that the function $\Pi(s, q)$ has at most one zero
$\rho=\beta+it$ in the region $0\leq\beta<1-1/R_0\log \max\{q, q\abs{t}\}$,
which must be real and simple and induced by some nonprincipal real primitive character $\wt\chi\pmod{\wt{q}}$ with $987\leq \wt{q}\leq x$.
Moreover, Theorem 1.3 of \cite{Kad} implies that, for any given $Q_1$, such zero satisfies $\beta<1-1/2R_1\log Q_1$
except possibly one modulus below $Q_1$.
Henceforth let $k_0$ be a such a modulus if it exists
and call this modulus and the corresponding character to be exceptional.  Furthermore,
we set $\dt=7/5$ and define $k_1=k_0$ if $k_0\geq\log^\dt x$ and $k_1=0$ otherwise,
so that $k_1\mid k$ is equivalent to both $k_0\mid k$ and $k_0\geq \log^\dt x$ hold.

\begin{cor}\label{cor231}
Assume that $x$ is a real number $>X_2=\exp\exp 32$, $x_2=\frac{e^{-100} x}{\log^4 x}$,
$K_0=\log^\dt x_2$ and let $Q_1=\log^{10} x_2$.
Moreover, let $k_1=k_0$ if $k_0\geq K_0$ and $k_1=0$ otherwise,
If $k$ is a modulus $\leq Q_1$ not divisible by $k_1$,
then we have
\begin{equation}
\quad \frac{\vph(k)}{x}\abs{E_\pi(x; k, l)}<\frac{e^{-14}}{\log^4 x}.
\end{equation}
\end{cor}

\begin{proof}
We begin by observing that $\beta_0\leq 1-\pi/0.4923 K_0^{1/2}\log^2 K_0$.
It is clear that either $k_0\nmid k$ or $k_0<\log^\dt x_2$ holds if there exists the Siegel zero.
In the case $k_0\nmid k$, it is clear that $\beta_0<1-1/2R_1\log Q_1=1-1/20R_1\log\log x_2<1-\pi/0.4923 K_0^{1/2}\log^2 K_0$.
In the other case $k_0<K_0$, Theorem 3 of \cite{LW} gives
$\beta_0\leq 1-\pi/0.4923k_0^{1/2}\log^2 k_0\leq 1-\pi/0.4923 K_0^{1/2}\log^2 K_0$.

Let $y$ be an arbitrary real number with $x_2<y\leq x$.  Since $1/2<\beta_0\leq 1-\pi/0.4923 K_0^{1/2}\log^2 K_0$,
we can see that $\frac{y^{\beta_0-1}}{\beta_0}<\frac{e^{-15}}{\log^3 y}$ from $x>X_2$
if the Siegel zero exists.
Thus, by Lemma \ref{lm23},
\begin{equation}
\frac{\vph(q)}{y}\abs{\psi(y; k, l)-\frac{y}{\vph(k)}}<\frac{1.1e^{-15}}{\log^3 y}.
\end{equation}
The rough estimate $\abs{\psi(y; k, l)-\theta(y; k, l)}<y^\frac{1}{2}\log^2 y/\log 2$ is enough to give
\begin{equation}
\frac{\vph(k)}{y}\abs{\theta(y; k, l)-\frac{y}{\vph(k)}}<\frac{1.2e^{-15}}{\log^3 y}.
\end{equation}
Since
\begin{equation}
\vph(k)\abs{\frac{\theta(y)}{y}-1}<\frac{0.1e^{-15}}{\log^3 y}
\end{equation}
for $y>X_2$ by Theorem 2 of \cite{RS2}, we have
\begin{equation}
\frac{\vph(k)}{y}\abs{E_\theta(y; k, l)}<\frac{1.3e^{-15}}{\log^3 y}.
\end{equation}

Now, partial summation yields
\begin{equation}
E_\pi(x; k, l)=E_\pi(x_2; k, l)+\frac{E_\theta(x; k, l)}{\log x}-\frac{E_\theta(x_2; k, l)}{\log x_2}+\int_{x_2}^x\frac{E_\theta(t; k, l)}{t\log^2 t}dt,
\end{equation}
and
\begin{equation}
\begin{split}
\vph(k)\int_{x_2}^x\frac{E_\theta(t; k, l)}{t\log^2 t}dt<&\frac{1.3e^{-15}}{1-\frac{5}{\log x_2}}\int_{x_2}^x\frac{1}{\log^5 t}\left(1-\frac{5}{\log t}\right)dt\\
=&\frac{1.3e^{-15}}{1-\frac{5}{\log x_2}}\left(\frac{x}{\log^5 x}-\frac{x_2}{\log^5 x_2}\right).
\end{split}
\end{equation}
Hence we obtain
\begin{equation}
\abs{E_\pi(x; k, l)}<\frac{1.3e^{-15} x}{\log^4 x}+\frac{2.6e^{-15} x_2}{\log^4 x_2}+1.4e^{-15}\left(\frac{x}{\log^5 x}-\frac{x_2}{\log^5 x_2}\right).
\end{equation}
The right-hand side does not exceed $\frac{e^{-14}}{\log^4 x}$ if $x>X_2$, which proves the corollary.
\end{proof}

\begin{lem}\label{lm24}
Assume that $x$ is a any real number $>X_2$.
Let $x_2, Q_1, k_1$ be as in Corollary \ref{cor231} and $Q=\frac{\sq{x_2}}{\log^{10} x_2}$.
Then we have
\begin{equation}
\sum_{k\leq Q, k_1 \nmid k}\mu^2(k)\max_{l\pmod{k}}\abs{E_\pi(x; k, l)}<\frac{e^{-8}x}{\log^3 x}.
\end{equation}
\end{lem}

\begin{proof}
Let $y$ be an arbitrary real number with $x_2<y\leq x$.
We begin by showing that
\begin{equation}\label{eq24}
\sum_{k\leq Q, k_1\nmid k}\mu^2(k)\max_{l\pmod{k}}\abs{E_\psi(y; k, l)}<\frac{1.9e^{-9}y}{\log^2 y}.
\end{equation}

As in the proof of Corollary \ref{cor231}, either $k_0\mid k$ or $k_0\geq K_0$ holds
if the Siegel zero exists.  Moreover, in the case $k_0\geq K_0$, we have $k_1=k_0$.

If $k_1=k_0$ or there exists no Siegel zero, we can apply Theorem 1.4 of \cite{Ymd1} with $A=10$
but $Q_1$ in this theorem replaced by $Q_1$ in Corollary \ref{cor231}.
Let $c_0, c_1, C$ be the constants defined by
\begin{equation}
\begin{split}
c_0= & \frac{2^{\frac{13}{2}}}{9\pi\log 2}\left(\frac{1}{3}+\frac{3}{2\log 2}\right)\left(\frac{2+\log(\log 2/\log(4/3))}{\log 2}\right)\sqrt\frac{\psi(113)}{113},
\end{split}
\end{equation}
\begin{equation}
c_1=\prod_p\left(1+\frac{1}{p(p-1)}\right)=\frac{\zeta(2)\zeta(3)}{\zeta(6)}
\end{equation}
and $C=0.0000128$, as in Theorem 1.4 of \cite{Ymd1}.  We note that $c_0<48.83215$ and $c_1<1.9436$.
Now, since $Q_1=\log^{10} x_2\leq\log^{10} y$, Theorem 1.4 of \cite{Ymd1} gives
\begin{equation}
\begin{split}
&\sum_{k\leq Q, k_1\nmid k}\mu^2(k)\max_{l\pmod{k}}\abs{E_\psi(y; k, l)}\\
<&\frac{c_0c_1(2+e^{-800})y}{\log^{\frac{11}{2}} y}+\frac{2c_0c_1y\log^\frac{9}{2} y}{Q_1}+\frac{c_1^2(C+e^{-17})y(1+10\log\log y)}{4\log^6 y}.
\end{split}
\end{equation}

Since we can see that $Q_1>(1-e^{-20})\log^{10} x$, we have
\begin{equation}
\begin{split}
&\sum_{k\leq Q, k_1\nmid k}\mu^2(k)\max_{l\pmod{k}}\abs{E_\psi(y; k, l)}\\
<&\frac{c_0c_1(4+e^{-19})y}{\log^{\frac{11}{2}} y}+\frac{c_1^2(C+e^{-17})y(1+10\log\log y)}{4\log^6 y}\\
<&\frac{379.64067y}{\log^{\frac{11}{2}} y}<\frac{1.9e^{-9}y}{\log^2 y}
\end{split}
\end{equation}
and therefore (\ref{eq24}) holds.

Next, consider the case $k_0<K_0$.  Now we see that
\begin{equation}
\begin{split}
&\sum_{k\leq Q}\mu^2(k)\max_{l\pmod{k}}\abs{E_\psi(y; k, l)}\\
< & \left(\sum_{1\leq m\leq Q}\frac{1}{\vph(m)}\right)\sum_{1<q\leq Q}\abs{\sideset{}{^*}\sum_{\chi\pmod{q}}\psi(y, \chi)}.
\end{split}
\end{equation}
Similarly to the proof of Theorem 1.4 of \cite{Ymd1}, we have
\begin{equation}
\sideset{}{^*}\sum_{\chi\pmod{k_0}}\abs{\psi(y, \chi)}<\frac{(C+e^{-17})y}{\log^6 y}+\frac{y^{\beta_0-1}}{\beta_0},
\end{equation}
where $\beta_0$ denotes the Siegel zero modulo $k_0$.
Theorem 3 of \cite{LW} gives $\beta_0\leq 1-\pi/0.4923k_0^{1/2}\log^2 k_0\leq 1-\pi/0.4923K_0^{1/2}\log^2 K_0$
and we see that
\begin{equation}
\frac{y^{1-\frac{\pi}{0.4923K_0^\frac{1}{2}\log^2 K_0}}}{1-\frac{\pi}{0.4923K_0^\frac{1}{2}\log^2 K_0}}<\frac{e^{-11.5}y}{\log^3 y\log\log y}.
\end{equation}
A similar argument to the proof of Theorem 1.4 of \cite{Ymd1}
using this inequality instead of (52) in \cite{Ymd1} gives (\ref{eq24}).
Similarly to (53), for each $k\leq Q_1$, we have
\begin{equation}
\sideset{}{^*}\sum_{\chi\pmod{k}}\abs{\psi(y, \chi)}<\frac{(C_0+e^{-18})y}{\log^7 y}+\frac{e^{-11.5}y}{\log^3 y\log\log y}<\frac{2e^{-12}y}{\log^3 y\log\log y}
\end{equation}
and, similarly to (54) in \cite{Ymd1}, we obtain
\begin{equation}
\sum_{q\leq Q_1, q_0\nmid q}\frac{1}{\vph(q)}\sideset{}{^*}\sum_{\chi\pmod{q}}\abs{\psi(x, \chi)}\leq\frac{2.01c_1e^{-12}y(1+10\log\log y)}{\log^3 y\log\log y}<\frac{1.95e^{-9}y}{\log^3 y}.
\end{equation}
This gives
\begin{equation}
\sum_{k\leq Q, k_1\nmid k}\mu^2(k)\max_{l\pmod{k}}\abs{E_\psi(y; k, l)}<\frac{1.95c_1e^{-9}y}{2\log^2 y}<\frac{1.9e^{-9}y}{\log^2 y}.
\end{equation}

Now we estimate $E_\pi(x; k, l)$.  We begin by observing that partial summation gives
\begin{equation}
E_\pi(x; k, l)=E_\pi(x_2; k, l)+\frac{E_\theta(x; k, l)}{\log x}-\frac{E_\theta(x_2; k, l)}{\log x_2}+\int_{x_2}^x\frac{E_\theta(t; k, l)}{t\log^2 t}dt.
\end{equation}
We would like to majorize the four terms in the right-hand side.

From the argument in the proof of Theorem A. 17 of \cite{Nat}, we see that $\sum_{k\leq Q}\frac{1}{\vph(k)}<c_1(1+\log Q)<\frac{c_1}{2}\log x$
and therefore (\ref{eq24}) yields
\begin{equation}\label{eq25}
\begin{split}
&\sum_{k\leq Q, k_1\nmid k}\mu^2(k)\max_{l\pmod{k}}\abs{E_\theta(y; k, l)}\\
<&\sum_{k\leq Q, k_1\nmid k}\mu^2(k)\max_{l\pmod{k}}\abs{E_\psi(y; k, l)}\\
&+\frac{2y^\frac{1}{2}\log^2 y}{\log 2} \sum_{k\leq Q, k_1\nmid k}\frac{\mu^2(k)}{\vph(k)}\\
<&\frac{1.9e^{-9}y}{\log^2 y}+\frac{c_1y^\frac{1}{2}\log^2 y\log x}{\log 2}\\
\end{split}
\end{equation}
for $x_2\leq y\leq x$.
Hence we obtain
\begin{equation}\label{eq261}
\sum_{k\leq Q}\mu^2(k)\max_{l\pmod{k}} \abs{E_\theta(x; k, l)}<\frac{2e^{-9}x}{\log^2 x},
\end{equation}
\begin{equation}\label{eq262}
\sum_{k\leq Q}\mu^2(k)\max_{l\pmod{k}} \abs{E_\theta(x_2; k, l)}<\frac{1.9e^{-9}x_2}{\log^2 x_2}+\frac{c_1x^\frac{1}{2}\log^3 x}{\log 2}
\end{equation}
and
\begin{equation}\label{eq263}
\begin{split}
\int_{x_2}^x\frac{\max_{l\pmod{k}} \abs{E_\theta(t; k, l)}}{t\log^2 t}dt<&\int_{x_2}^x\frac{1.9e^{-9}}{\log^4 t}+\frac{c_1\log x}{t^\frac{1}{2}\log 2}dt\\
<&\frac{1}{1-\frac{1}{5\log x_2}}\int_{x_2}^x\frac{1.9e^{-9}}{\log^4 t}\left(1-\frac{4}{\log t}\right)dt+2c_1x^\frac{1}{2}\log x\\
=&\frac{1.9e^{-9}}{1-\frac{1}{5\log x_2}}\left(\frac{x}{\log^4 x}-\frac{x_2}{\log^4 x_2}\right)+2c_1x^\frac{1}{2}\log x\\
<&\frac{2e^{-9}x}{\log^4 x}.
\end{split}
\end{equation}

Moreover, we use a trivial estimate $E_\pi(x_2; k, l)\leq x_2$ to obtain
\begin{equation}\label{eq264}
\sum_{k\leq Q}\mu^2(k)\max_{l\pmod{k}}E_\pi(x_2; k, l)<x_2\sum_{k\leq Q}\frac{\mu^2(k)}{\vph(k)}<c_1 x_2\log x.
\end{equation}

Combining (\ref{eq261})-(\ref{eq264}), we have
\begin{equation}\label{eq265}
\sum_{k\leq Q, k_1 \nmid k}\mu^2(k)\max_{l\pmod{k}}\abs{E_\pi(x; k, l)}<\frac{e^{-8}x}{\log^3 x}.
\end{equation}
This proves the lemma.
\end{proof}

Now we introduce a extention of the previous lemma, which plays an important role in our argument
to avoid the exceptional modulus.
\begin{lem}\label{lm25}
Let $x>X_2=\exp\exp 36$ and $Q, k_1$ as in the previous lemma.  If $k$ divides $k_1$, 
then we have
\begin{equation}
\begin{split}
&\sideset{}{^\#}\sum_{d\leq Q/k} \mu^2(k)\max_{l\pmod{k}}\abs{\pi(x; kd, l)-\frac{\pi(x; k, l)}{\vph(d)}}\\
&\quad <\frac{e^{-8}x}{\log^3 x},
\end{split}
\end{equation}
where $d$ runs over integers such that $k_1\nmid kd$ if $k\neq k_1$. 
\end{lem}
\begin{proof}
We have
\begin{equation}
\psi(x; kd, l)-\frac{\psi(x; k, l)}{\vph(d)}=\sum_{\chi_1, \chi_2}\frac{1}{\vph(kd)}\sum_{\substack{\chi_1\pmod{k},\\ \chi_2\neq \chi_{0, d}\pmod{d}}}\bar\chi_1(l)\bar\chi_2(l)\psi(x; \chi_1\chi_2),
\end{equation}
where $\chi_{0, d}$ denotes the trivial character modulo $d$,
and observe that there exists no character $\chi_1\chi_2$ appearing in this sum induced from the exceptional one $\pmod{k_1}$ since $\chi_2$ is nontrivial
and either $k_1\nmid kd$ or $k=k_1$ holds.
Now, similarly to (\ref{eq24}), we have
\begin{equation}
\begin{split}
 &\sideset{}{^\#}\sum_{d\leq Q/k} \mu^2(q)\max_{l\pmod{kd}}\abs{\psi(y; kd, l)-\frac{\psi(y; k, l)}{\vph(d)}}\\
<&\frac{1.9e^{-9}y}{\log^3 y}
\end{split}
\end{equation}
for $x_2<y\leq x$.
The remaining argument essentially repeats the proof of the previous lemma.
\end{proof}

\section{An explicit Rosser-Iwaniec linear sieve}\label{RI}

In this section, we introduce an explicit version of Rosser-Iwaniec linear sieve.
We use the following notation: $A$ is a finite set of integers,
$\Omega_p$ a set of congruent classes modulo $p$
and $\rho(p)$ be a multiplicative arithmetic function which takes zero if $\Omega_p$ is empty,
\begin{equation*}
\begin{split}
A_p=& A\cap \Omega_p, A_d=\bigcap_{p\mid d}A_p, r(d)=\abs{A_d}-\frac{\rho(d)}{d}\abs{A},\\
S(A, P)=& \abs{A-\bigcup_{p\mid P} A_p}, V(P)=\prod_{p\mid P}\left(1-\frac{\rho(p)}{p}\right)\\
\intertext{and}
P(z)=& \prod_{p<z}p.
\end{split}
\end{equation*}

Now we can state that the sieve problem is to estimate $S(A, P)$ under the condition that $r(d)$ is small.
A special case is the case $\prod_{p\leq x} (1-\rho(p)/p)\sim C\log x$ for some constant $C$,
which is called linear.

Using Selberg's sieve, Jurkat and Richert\cite{JR} gave upper and lower bounds for the linear sieve.
Using Rosser's combinatorial argument in his unpublished manuscript, Iwaniec\cite{Iwa} improved their upper and lower bounds.
Moreover, an explicit version of Rosser-Iwaniec linear sieve is given in Chapter 9 in \cite{Nat} although
it requires an additional condition.
In \cite{Ymd2}, the author gave another explicit version of Rosser-Iwaniec linear sieve which can be
applied in more general cases.  Here we shall introduce Theorem 1.2 in \cite{Ymd2}.
The assumption in this theorem is satisfied, for instance, when $A$ is the set of odd integers
of the form $aq+b$ with $a, b$ fixed coprime integers and $q$ odd prime
and $\Omega_p$ consists at most one congruent class modulo $p$ for each prime $p$.
In particular, $A$ can be taken to be sets of integers the form $N-q$
with $N$ even and $q$ odd prime, which we shall consider in the following sections.

\begin{thm}\label{th31}
Assume that $\rho(p)\leq p/(p-1)$ and $\rho(2)=0$.
Then, for every $D, s>0$, we have
\begin{equation}
S(A, P(D^\frac{1}{s}))>X\left(V(P(D^\frac{1}{s}))-2\left(\frac{f_1(s)}{\log D}+\frac{255.84406}{\log^\frac{3}{2} D}\right)\right)-\abs{R(D, P(z))}
\end{equation}
and
\begin{equation}
S(A, P(D^\frac{1}{s}))<X\left(V(P(D^\frac{1}{s}))+2\left(\frac{F_1(s)}{\log D}+\frac{298.87013}{\log^\frac{3}{2} D}\right)\right)+\abs{R(D, P(z))}.
\end{equation}
where $f_1(s), F_1(s)$ are functions such that
$F_1(s)=2e^\gamma-s$ for $0\leq s\leq 3$ and
\begin{equation}
\begin{split}
F_1^\prime(s)=& -\frac{f(s-1)}{s-1}\text{ for }s\geq 3,\\
f_1^\prime(s)=& -\frac{F_1(s-1)}{s-1}\text{ for }s\geq 2,
\end{split}
\end{equation}
and $R(D, P)=\sum_{d\mid P, d\leq D}\abs{\mu^2(d)r(d)}$.
\end{thm}

\section{Framework of our sieve argument}

We use the following notation.
As we assumed in Theorem \ref{th1}, let $N$ be an even integer $\geq X_2=\exp\exp 36$.
Let $z=N^\frac{1}{8}$ and $y=N^\frac{1}{3}$.  Then $z>\exp\exp 33$ and $y>\exp\exp 34$.

We shall consider the set $A=\{N-p: p\leq N, p\nmid N\}$.
If $A$ contains at least one prime, then $N$ could be represented by the sum of two primes.
We set $\Omega_p$ to be the congruent class $0\pmod{p}$, so that $A_q=\{m: q\mid m\}$.
Moreover, let $r(d)=\abs{A_{d}}-\frac{\abs{A}}{\vph(d)}$ and $r_k(d)=\abs{A_{kd}}-\frac{\abs{A_k}}{\vph(d)}$
denote error terms.
Clearly we have $\abs{A}=\pi(N)-\omega(N)$ and $\abs{A_k}=\pi(N; k, N)-\omega(N; k, N)$, where $\omega(n; q, a)$
denotes the number of prime factors of $n$ which is equivalent to $a\pmod{q}$.

As Chen and other authors did, we introduce the other set
$B=\{N-p_1p_2p_3: z\leq p_1<y\leq p_2\leq p_3, p_1p_2p_3<N, (p_1p_2p_3, N)=1\}$
and obtain the following lower bound, which is Theorem 10.2 in \cite{Nat}.
\begin{lem}\label{lm41}
\begin{equation}
\pi_2(N)> S(A, P(z))-\frac{1}{2}\sum_{z\leq q<y}S(A_q, P(z))-\frac{1}{2}S(B, P(y))-2N^\frac{7}{8}-2N^\frac{1}{3}.
\end{equation}
\end{lem}

So that, it suffices to give an lower bound for $S(A, P(z))$ and
upper bounds for $S(B, P(y))$ and $S(A_q, P(z))$ for each primes $q$ with $z\leq q<y$.

We set $x_2=\frac{e^{-100} N}{\log^4 N}$,
which coincides $x_2$ in Section \ref{SecPre} with $x=N$,
$K_0=\log^\dt x_2$ and let $Q_1=\log^{10} x_2$.
Moreover, we set $k_0$ to be the exceptional modulus defined as in Section \ref{SecPre}
and $k_1=k_0$ if $k_0\geq K_0$ and $k_1=0$ otherwise,

Let $q_1> q_2>\cdots> q_l$ be all prime factors of $k_1$
and $m_j=q_1q_2\cdots q_j, A^{(j)}=A_{m_j}$ and $P^{(j)}(x)=\prod_{p<x, p\nmid N, p\neq q_1, q_2, \ldots, q_j}p$
for $j=0, 1, 2, \ldots, l$.  We note that $m_0=1, A^{(0)}=A, P^{(0)}(x)=P(x)$.
Moreover, we write for brevity $V(x)=V(P(x))$ amd $V^{(j)}(x)=V(P^{(j)}(x))$.

As in the proof of Theorem 10.3 in \cite{Nat}, we deduce from (\ref{eq22}) that
\begin{equation}
V^{(j)}(x)=\frac{U^{(j)}_N}{\log x}\left(1+\frac{\theta}{5\log^2 x}\right)\left(1+\frac{8\theta\log x}{x}\right)
\end{equation}
for $j=0, 1, \ldots, l$ and $x\geq z$, where
\begin{equation}
U_N^{(j)}=2e^{-\gamma}\prod_{p>2}\left(1-\frac{1}{(p-1)^2}\right)\prod_{p>2, p\mid Nm_j}\frac{p-1}{p-2},
\end{equation}
so that $U_N^{(0)}=U_N$.
Moreover, we have
\begin{equation}\label{eq40a}
l<\frac{1.3841 \log(10\log N)}{\log\log(10 \log N)}<\frac{e\log\log N}{\log\log\log N}
\end{equation}
by (\ref{eq22}). Moreover, since $k_1\geq \log^\dt x_2>3\times 5\times 7\times 11\times \cdots \times 53$, we have
\begin{equation}
q_1\geq \max\{59, 2+\log\log N\}
\end{equation}
and therefore 
\begin{equation}\label{eq40b}
U_N^{(1)}\leq U_N\frac{q-1}{q-2}\leq U_N(1+\ep_0(N)),
\end{equation}
where $\ep_0(N)=\frac{1}{\max\{57, \log\log N\}}$.

We can easily see that, for any $d\mid P^{(j)}(x)$, $r_{m_j}(d)=\abs{A^{(j)}_d}-\frac{\abs{A^{(j)}}}{\vph(d)}$.
We have the following estimates.
\begin{lem}
We have
\begin{equation}\label{eq41a}
\abs{A}>\frac{N}{\log N}.
\end{equation}
For $j=0, 1, \ldots, l-1$, we have
\begin{equation}\label{eq41b}
\abs{r(m_j)}<\frac{e^{-8}N}{\log^3 N}
\end{equation}
and
\begin{equation}\label{eq41c}
\abs{r(m_l)}<\frac{0.19N}{\log^{2.3} N}.
\end{equation}
Moreover, for any integer $k$ dividing $k_1$, we have
\begin{equation}\label{eq41d}
\sideset{}{^\#}\sum_{d<Q/k}\mu^2(kd)\abs{r_{k}(d)}<\frac{1.1e^{-8}N}{\log^3 N},
\end{equation}
where $Q=\frac{x_2}{\log^{10} x_2}$ and $d$ runs over integers such that $k_1\nmid kd$ if $k\neq k_1$.
\end{lem}
\begin{proof}
We recall that $\abs{A}=\pi(N)-\omega(N)$ and
(\ref{eq41a}) easily follows from (\ref{eq21}) and Theorem 1, (3.1) of \cite{RS1} or Theorem 6.9, (6.5) of \cite{Dus}.

We observe that Corollary \ref{cor231} gives (\ref{eq41b}) for $j=0, 1, \ldots, l-1$.
Moreover, Brun-Titchmarsh's inequality in the form \cite[Theorem 2]{MV} gives
\begin{equation}
\begin{split}
\abs{r(m_l)}<& \left(\frac{2\log N}{\log(N/m_l)}-1\right)\frac{N}{\vph(m_l)\log N}+\omega(N)\\
<& (1+e^{-20})\frac{N(e^\gamma\log(\log\log N+\log\dt)+0.5)}{\log^{1+\dt} N}\\
<& \frac{0.19N}{\log^{2.3} N}
\end{split}
\end{equation}
since $m_l=k_1>K_0=\log^\dt N\geq \log^\dt X_2$ and $\vph(m_l)>m_l/(e^\gamma\log\log m_l+0.5)$ by Theorem 15, (3.41-42) of \cite{RS1}.
This proves (\ref{eq41c}).

Recalling that $\abs{A_k}=\pi(N; k, N)-\omega(N; k, N)$ and $\abs{A_{kd}}=\pi(N; kd, N)-\omega(N; kd, N)$,
we see that
\begin{equation}
r_k(d)=\pi(N; kd, N)-\frac{\pi(N; k, N)}{\vph(d)}+\frac{1.3841\theta\log N}{\log\log N}
\end{equation}
and therefore
\begin{equation}
\sideset{}{^\#}\sum_{d<Q/k}\mu^2(kd)\abs{r_{k}(d)}
<\frac{1.3841Q\log N}{\log\log N}+\sideset{}{^\#}\sum_{d<Q/k}\abs{\pi(N; kd, N)-\frac{\pi(N; k, N)}{\vph(d)}}.
\end{equation}
Now we apply Lemma \ref{lm25} with $x=N$ and obtain
\begin{equation}
\sideset{}{^\#}\sum_{d<Q/k}\mu^2(kd)\abs{r_{k}(d)}<\frac{1.3841D\log N}{\log\log N}+\frac{e^{-8}x}{\log^3 x}<\frac{1.1e^{-8}x}{\log^3 x}.
\end{equation}
\end{proof}

\section{Lower bounds for some sums over primes}\label{LS}

The purpose of this section is to obtain an lower bound for $S(A, P(z))$:
\begin{thm}\label{th4}
\begin{equation}
\begin{split}
S(A, P(z))>& \frac{U_N\abs{A}}{\log N}\\
& \times \left(4e^\gamma\log 3-0.5198\ep_0(N)-\frac{767.7471}{\log^\frac{1}{2} N}\right).
\end{split}
\end{equation}
\end{thm}

As mentioned in the introduction, we cannot directly estimate $S(A, P(z))$ due to possible existence of the exceptional zero $k_1$.
However, the following inclusion-exclusion identity allows us to overcome this obstacle.
\begin{lem}\label{lm42}
\begin{equation}
S(A, P(z))=\sum_{i=0}^{l-1}(-1)^i S(A^{(i)}, P^{(i+1)}(z))+(-1)^l S(A^{(l)}, P^{(l)}(z)).
\end{equation}
\end{lem}
\begin{proof}
$S(A, P^{1}(z))-S(A, P(z))$ counts the number of 
integers in $A$ divisible by $q_1$ but not divisible by any other primes below $z$.
$S(A^{(1)}, P^{(2)}(z))-S(A, P^{1}(z))+S(A, P(z))$ counts the number of 
integers in $A$ divisible by $q_1, q_2$ but not divisible by any other primes below $z$.
Iterating this argument, we see that 
$\sum_{i=0}^{l-1}(-1)^{l-1-i} S(A^{(i)}, P^{(i+1)}(z))+(-1)^l S(A, P(z))$
counts the number of integers in $A$ divisible by $q_1, q_2, \ldots, q_l$
but not divisible by any other primes below $z$,
which is equal to $S(A^{(l)}, P^{(l)}(z))$.
\end{proof}

As we will see below, each quantity can be estimated by sieve argument without encountering
the exceptional character.

Theorem \ref{th31} immediately gives the following estimates:
\begin{lem}\label{lm43}
Let
\begin{equation}
E_j=
\begin{cases}\sum_{d\mid P^{(j+1)}(z), d<D}r_{m_j}(d) &\text{if }j=0, 1, \ldots, l-1\\
\sum_{d\mid P^{(l)}(z), d<D}r_{m_l}(d) &\text{if } j=l.
\end{cases}
\end{equation}
Then we have
\begin{equation}
\begin{split}
& S(A^{(j)}, P^{(j+1)}(z))\\
>& \abs{A^{(j)}}\left[V^{(j+1)}(z)-U_N^{(j+1)}\left(\frac{f(s_j)}{\log D}+\frac{255.84406}{\log^\frac{3}{2} D}\right) \right]-\abs{E_j}
\end{split}
\end{equation}
and
\begin{equation}
\begin{split}
& S(A^{(j)}, P^{(j+1)}(z))\\
<& \abs{A^{(j)}}\left[V^{(j+1)}(z)+U_N^{(j+1)}\left(\frac{F(s_j)}{\log D}+\frac{298.87013}{\log^\frac{3}{2} D}\right) \right]+\abs{E_j}
\end{split}
\end{equation}
for $j=0, 1, \ldots, l-1$.  Moreover, we have
\begin{equation}
\begin{split}
& S(A^{(l)}, P^{(l)}(z))\\
>& \abs{A^{(l)}}\left[V^{(l)}(z)-U_N^{(l)}\left(\frac{f(s_l)}{\log D}+\frac{255.84406}{\log^\frac{3}{2} D}\right) \right]-\abs{E_l}
\end{split}
\end{equation}
and
\begin{equation}
\begin{split}
& S(A^{(l)}, P^{(l)}(z))\\
<& \abs{A^{(l)}}\left[V^{(l)}(z)+U_N^{(l)}\left(\frac{F(s_l)}{\log D}+\frac{298.87013}{\log^\frac{3}{2} D}\right) \right]+\abs{E_l}
\end{split}
\end{equation}
\end{lem}

Let $D=\frac{\sqrt{x_2}}{k_1\log^{10} x_2}$ and $s_j=\frac{\log D/m_j}{\log z}$.
Since $m_j<k_1<\log^{10} N$, we have
\begin{equation}\label{eq42a}
\log D>\frac{\log N}{2}-22\log\log N-50
\end{equation}
and
\begin{equation}\label{eq42b}
4-\frac{8(32\log\log N+50)}{\log N}<s_1<4.
\end{equation}

We majorize $\abs{E_j}$ for each $j=0, 1, \ldots, l-1$ as well as for $j=l$.
It is almost trivial that $D<Q$ and therefore (\ref{eq41d}) immediately gives
\begin{equation}
\abs{E_j}<\frac{1.3841D\log N}{\log\log N}+\frac{1.1e^{-8}N}{\log^3 N}<\frac{1.2e^{-8}N}{\log^3 N}
\end{equation}
for each $j=0, 1, \ldots, l-1$ as well as $j=l$.

Substituting this estimate into the inequalities in Lemma \ref{lm43} gives
\begin{equation}\label{eq43a}
\begin{split}
& S(A^{(j)}, P^{(j+1)}(z))\\
>& \abs{A^{(j)}}\left[V^{(j+1)}(z)-U_N^{(j+1)}\left(\frac{f(s_j)}{\log D}+\frac{255.84406}{\log^\frac{3}{2} D}\right) \right]\\
& -\frac{1.2e^{-8}N}{\log^3 N}.
\end{split}
\end{equation}
and
\begin{equation}\label{eq43b}
\begin{split}
& S(A^{(j)}, P^{(j+1)}(z))\\
<& \abs{A^{(j)}}\left[V^{(j+1)}(z)+U_N^{(j+1)}\left(\frac{F(s_j)}{\log D}+\frac{298.87013}{\log^\frac{3}{2} D}\right) \right]\\
& +\frac{1.2e^{-8}N}{\log^3 N}.
\end{split}
\end{equation}
Similar estimates also holds for $S(A^{(l)}, P^{(l)}(z))$.

We see that $f(s_j), F(s_j)<0.0866$ for our values $s_j$ since $s_j>4-\frac{11\log\log N}{\log N}>3.9999$.
Hence, substituting (\ref{eq43a}) and (\ref{eq43b}) (and similar estimates for $S(A^{(l)}, P^{(l)}(z))$)
into Lemma \ref{lm42}, using (\ref{eq40a}) and observing that
the error terms are
\begin{equation}
\frac{1.2e^{-8}N\log\log N}{\log^3 N\log\log\log N}<\frac{e^{-23}N}{\log^{2.5} N},
\end{equation}
we obtain
\begin{equation}\label{eq44}
\begin{split}
S(A, P(z))>&\sum_{j=0}^{l-1}(-1)^j\abs{A^{(j)}}V^{(j+1)}(z)+(-1)^l\abs{A^{(l)}}V^{(l)}(z)\\
&-\abs{A}U_N^{(1)}\left(\frac{f(s_1)}{\log D}+\frac{255.84406}{\log^\frac{3}{2} D}\right)\\
&-\left(\sum_{j=1}^{l-1}\abs{A^{(j)}}U_N^{(j+1)}+\abs{A^{(l)}}U_N^{(l)}\right)\\
&\quad \times \left(\frac{0.0866}{\log D}+\frac{298.87013}{\log^\frac{3}{2} D}\right)\\
&-\frac{e^{-23}N}{\log^{2.5} N}.
\end{split}
\end{equation}

Now we shall evaluate each line in (\ref{eq44}).  We shall begin by showing that
\begin{equation}\label{eq45}
\sum_{j=0}^{l-1}(-1)^j\abs{A^{(j)}}V^{(j+1)}(z)+(-1)^l\abs{A^{(l)}}V^{(l)}(z)
=\abs{A}V(z)\left(1+\frac{e^{-27}\theta}{\log^\frac{1}{2} N}\right).
\end{equation}

We divide each term in the left-hand side of (\ref{eq45}) and obtain
\begin{equation}\label{eq46}
\begin{split}
& \sum_{j=0}^{l-1}(-1)^j\abs{A^{(j)}}V^{(j+1)}(z)+(-1)^l\abs{A^{(l)}}V^{(l)}(z)\\
=& \sum_{j=0}^{l-1}(-1)^j\frac{\abs{A}}{\vph(m_j)}V^{(j+1)}(z)+(-1)^l\frac{\abs{A}}{\vph(m_l)}V^{(l)}(z)\\
& +\sum_{j=0}^{l-1}(-1)^jr(m_j)V^{(j+1)}(z)+(-1)^l r(m_l)V^{(l)}(z).
\end{split}
\end{equation}
We write $\vph^*(N)=N\prod_{p\mid N}\frac{p-2}{p}$.  Since
\begin{equation}
\frac{\abs{A}}{\vph(m_j)}V^{(j+1)}(z)=\frac{\abs{A}V(z)}{\vph^*(m_j)}\left(1+\frac{1}{q_{j+1}-2}\right),
\end{equation}
we have
\begin{equation}
\sum_{j=0}^{l-1}(-1)^j\frac{\abs{A}}{\vph(m_j)}V^{(j+1)}(z)=\abs{A}V(z)\left(1+\frac{(-1)^j}{\vph^*(m_{j+1})}\right)
\end{equation}
and therefore
\begin{equation}\label{eq47}
\sum_{j=0}^{l-1}(-1)^j\frac{\abs{A}}{\vph(m_j)}V^{(j+1)}(z)+(-1)^l\frac{\abs{A}}{\vph(m_l)}V^{(l)}(z)=\abs{A}V(z).
\end{equation}
Substituting (\ref{eq47}), (\ref{eq41b}), (\ref{eq41c}) into (\ref{eq46}), we obtain
\begin{equation}\label{eq45a}
\begin{split}
& \sum_{j=0}^{l-1}(-1)^j\abs{A^{(j)}}V^{(j+1)}(z)+(-1)^l\abs{A^{(l)}}V^{(l)}(z)\\
=& \abs{A}V(z)+\frac{0.193\theta N}{\log^{2.3} N}V^{(l)}(z)\\
\end{split}
\end{equation}

Using Theorem 15, (3.41-42) of \cite{RS1} again, we have
\begin{equation}
\begin{split}
\frac{V^{(l)}(z)}{V(z)}\leq\prod_{j=1}^{l}\frac{q_j-1}{q_j-2}<& \prod_{j=1}^{l}\frac{q_j}{q_j-1}\prod_p\frac{(q_j-1)^2}{q_j(q_j-2)}\\
<& 1.51479\left(e^\gamma\log\log k_1+\frac{5}{2\log\log k_1}\right)\\
<& 4\log\log\log N.
\end{split}
\end{equation}
Hence, (\ref{eq41a}) and (\ref{eq45a}) gives
\begin{equation}
\begin{split}
& \sum_{j=0}^{l-1}(-1)^j\abs{A^{(j)}}V^{(j+1)}(z)+(-1)^l\abs{A^{(l)}}V^{(l)}(z)\\
=& \abs{A}V(z)+\frac{0.772\theta N\log\log\log N}{\log^{2.3} N}V(z)\\
=& \abs{A}V(z)+\frac{e^{-27}\theta N}{\log^{1.5} N}V(z)\\
=& \abs{A}V(z)\left(1+\frac{e^{-27}\theta}{\log^\frac{1}{2} N}\right),
\end{split}
\end{equation}
which is (\ref{eq45}).

Next we shall estimate the second line of (\ref{eq44}). 
Since Lemma \ref{lm22} gives
\begin{equation}
V(z)-\frac{U_N f(s)}{\log D}>U_N\left(\frac{2e^\gamma\log(s-1)}{\log D}-\frac{1}{4\log^3 z}\right),
\end{equation}
We see that $\log D>(0.5-e^{-16})\log N$ and $\log(s_1-1)>\log 3-\frac{1}{e^9\log^\frac{1}{2}N}$
by (\ref{eq42a}) and (\ref{eq42b}).  Thus we have
\begin{equation}\label{eq48a}
\begin{split}
&U_N^{(1)}\left(\frac{f(s_1)}{\log D}+\frac{255.84406}{\log^\frac{3}{2} D}\right)\\
<& V(z)+U_N\left(-\frac{2e^\gamma\log(s_1-1)}{\log D}+\frac{1}{4\log^3 z}+
\frac{\ep_0(N)f(s_1)}{\log D}+\frac{255.84406(1+\ep_0(N))}{\log^\frac{3}{2} D}\right)\\
<& V(z)+U_N\left(-\frac{2e^\gamma\log(s_1-1)}{\log D}+\frac{128}{\log^3 N}+
\frac{0.1733\ep_0(N)}{\log N}+\frac{736.33191}{\log^\frac{3}{2} N}\right)\\
<& V(z)+U_N\left(-\frac{4e^\gamma\log 3-0.1733\ep_0(N)}{\log N}+\frac{736.33192}{\log^\frac{3}{2} N}\right).
\end{split}
\end{equation}

Finally we shall evaluate the term in the third and the fourth line of (\ref{eq44}).
With the aid of (\ref{eq41a}) we have
\begin{equation}\label{eq49}
\begin{split}
& \sum_{j=1}^{l-1}\abs{A^{(j)}}U_N^{(j+1)}+\abs{A^{(l)}}U_N^{(l)}\\
&\qquad <2U_N\ep_0(N)\abs{A}+\left(\frac{e^{-8}N}{\log^3 N}+\frac{0.19N}{\log^{2.3} N}\right)U_N^{(l)}\\
&\qquad <U_N\left(2\ep_0(N)\abs{A}+\frac{0.191N}{\log^{2.3} N}\prod_{p\mid k_1}\frac{p-1}{p-2}\right)\\
&\qquad <U_N\left(2\ep_0(N)\abs{A}+\frac{N\log\log\log N}{\log^{2.3} N}\right)\\
&\qquad <U_N\abs{A} \left(2\ep_0(N)+\frac{e^{-9}}{\log N}\right)\\
\end{split}
\end{equation}
observing that
\begin{equation}
\begin{split}
& U_N\left(\sum_{j=1}^{l}\frac{q_j-1}{\vph^*(m_{j-1})(q_j-2)}+\frac{1}{\vph^*(k_1)}\right)\\
&\qquad =\frac{1}{q_1-2}\left(1+\frac{1}{q_2-2}+\frac{1}{(q_2-2)(q_3-2)}+\cdots \right)\\
&\qquad\leq\frac{2U_N}{q_1-2}\\
&\qquad<2U_N\ep_0(N).
\end{split}
\end{equation}
Since it follows from (\ref{eq42a}) that
\begin{equation}
\frac{0.0866}{\log D}+\frac{298.87013}{\log^\frac{3}{2} D}<\frac{0.17321}{\log N}+\frac{845.33239}{\log^\frac{3}{2} N},
\end{equation}
we conclude that
\begin{equation}\label{eq48b}
\begin{split}
& \left(\sum_{j=1}^{l-1}\abs{A^{(j)}}U_N^{(j+1)}+\abs{A^{(l)}}U_N^{(l)}\right)\left(\frac{0.0866}{\log D}+\frac{298.87013}{\log^\frac{3}{2} D}\right)\\
& <U_N\abs{A}\ep_0(N)\left(\frac{0.3465}{\log N}+\frac{1790.6648}{\log^\frac{3}{2} N}\right).
\end{split}
\end{equation}
Substituting (\ref{eq45}), (\ref{eq48a}) and (\ref{eq48b}) into (\ref{eq44}), we obtain Theorem \ref{th4}.

\section{Upper bounds for some sums over primes}\label{US}

In this section, we shall obtain an upper bound for $\sum_{z\leq q<y}S(A_q, P, z)$:
\begin{thm}\label{th5}
\begin{equation}
\begin{split}
\sum_{z\leq q<y}S(A_q, P(z))<& \frac{U_N\abs{A}}{\log N}\\
&\times \left(4e^\gamma\log 6(1+\ep_0(N))+\frac{993.2507}{\log^\frac{1}{2} N}\right).
\end{split}
\end{equation}
\end{thm}

In this mission, it is much easier to break the obstacle due to possible existence of exceptional modulus;
it suffices to give an upper bound for $\sum_{z\leq q<y}S(A_q, P^{(1)}(z))$
since it is clear that $\sum_{z\leq q<y}S(A_q, P(z))\leq \sum_{z\leq q<y}S(A_q, P^{(1)}(z))$.

In this section, we set $D=\frac{x_2^\frac{1}{2}}{\log^{10} x_2}$ and $s_q=\frac{\log (D/q)}{\log z}$.
Theorem \ref{th31} immediately gives
\begin{equation}
\begin{split}
S(A_q, P^{(1)}(z))<& \abs{A_q}\left[V^{(1)}(z)+U_N^{(1)}\left(\frac{F(s_q)}{\log\frac{D}{q}}+\frac{298.87013}{\log^\frac{3}{2}\frac{D}{q}}\right) \right]\\
& +\sum_{d<D/q, d\mid P^{(1)}(z)}\abs{r_q(d)},
\end{split}
\end{equation}
where $r_q(d)=\abs{A_{qd}}-\abs{A_q}/\vph(d)$, and therefore the sum $\sum_{z\leq q<y}S(A_q, P, z)$ can be bounded from above by
\begin{equation}\label{eq51}
\begin{split}
&\sum_{z\leq q<y}\abs{A_q}\left[V^{(1)}(z)+U_N^{(1)}\left(\frac{F(s_q)}{\log\frac{D}{q}}+\frac{298.87013}{\log^\frac{3}{2}\frac{D}{q}}\right) \right]\\
& +\sum_{z\leq q<y}\sum_{d<D/q, d\mid P^{(1)}(z)}\abs{r_q(d)}
\end{split}
\end{equation}

Using Lemma \ref{lm24}, we obtain that the sum over the error terms is
\begin{equation}\label{eq52}
\begin{split}
&\sum_{z\leq q<y}\sum_{d<D/q, d\mid P^{(1)}(z)}\abs{r_q(d)}\\
\leq & \sum_{z\leq q<y}\sum_{d<D/q, d\mid P^{(1)}(z)}\left(\abs{\pi(N; qd, N)-\frac{\pi(N; q, N)}{\vph(d)}}+\omega(N)\right)\\
\leq & \sum_{z\leq q<y}\sum_{d<D/q, d\mid P^{(1)}(z)}\left(\abs{E_\pi(N; qd, N)}+\abs{\frac{E_\pi(N; q, N)}{\vph(d)}}+\omega(N)\right)\\
< & \frac{1.2e^{-8}N}{\log^3 N}.
\end{split}
\end{equation}

Since $N^\frac{1}{8}<\frac{D}{y}<\frac{D}{q}<\frac{D}{z}<N^\frac{3}{8}$, we have $1<s_q<3$
and therefore $F(s_q)=2e^\gamma-s_q$.
Thus we have
\begin{equation}
\begin{split}
V^{(1)}(z)+U_N^{(1)}\frac{F(s_q)}{\log\frac{D}{q}}=& U_N^{(1)}\frac{2e^\gamma}{\log\frac{D}{q}}+V^{(1)}(z)-\frac{U_N^{(1)}}{\log z}\\
<& U_N^{(1)}\left(\frac{2e^\gamma}{\log\frac{D}{q}}+\frac{1}{5\log^2 z}\right).
\end{split}
\end{equation}
This gives
\begin{equation}\label{eq53}
\begin{split}
& \sum_{z\leq q<y}\abs{A_q}\left[V^{(1)}(z)+U_N^{(1)}\left(\frac{F(s_q)}{\log\frac{D}{q}}+\frac{298.87013}{\log^\frac{3}{2}\frac{D}{q}}\right) \right]\\
& \leq U_N^{(1)}\sum_{z\leq q<y}\abs{A_q}\left(\frac{2e^\gamma}{\log\frac{D}{q}}+\frac{298.87013}{\log^\frac{3}{2}\frac{D}{q}}+\frac{1}{5\log^2 z}\right).
\end{split}
\end{equation}

Since $\abs{A_q}\leq \frac{\abs{A}+\omega(N)}{q-1}+E_\pi(N; q, N)$, we have
\begin{equation}\label{eq54}
\begin{split}
& \sum_{z\leq q<y}\abs{A_q}\left(\frac{2e^\gamma}{\log\frac{D}{q}}+\frac{298.87013}{\log^\frac{3}{2}\frac{D}{q}}+\frac{1}{5\log^2 z}\right)\\
& \leq \frac{z(\abs{A}+\omega(N))}{z-1}\sum_{z\leq q<y}\frac{2e^\gamma}{q\log\frac{D}{q}}+\frac{298.87013}{q\log^\frac{3}{2}\frac{D}{q}}+\frac{1}{5q\log^2 z}\\
&\quad +\left(\frac{2e^\gamma}{\log\frac{D}{y}}+\frac{298.87013}{\log^\frac{3}{2}\frac{D}{y}}+\frac{1}{5\log^2 z}\right)\sum_{z\leq q<y}E_\pi(N; q, N).
\end{split}
\end{equation}

Using Lemma \ref{lm24}, we have $\sum_{z\leq q<y}E_\pi(N; q, N)<\frac{e^{-8}N}{\log^3 N}$
and therefore (\ref{eq51}) is at most
\begin{equation}\label{eq55}
\begin{split}
&\sum_{z\leq q<y}\left(\frac{\abs{A}+\omega(N)}{\vph(q)}+E_\pi(N; q, N)\right)\left[V^{(1)}(z)+U_N^{(1)}\left(\frac{F(s_q)}{\log\frac{D}{q}}+\frac{298.87013}{\log^\frac{3}{2}\frac{D}{q}}\right) \right]\\
&<U_N^{(1)}\sum_{z\leq q<y}\left(\frac{\abs{A}+\omega(N)}{\vph(q)}+E_\pi(N; q, N)\right)\left(\frac{2e^\gamma}{\log\frac{D}{q}}+\frac{298.87013}{\log^\frac{3}{2}\frac{D}{q}}+\frac{1}{5\log^2 z}\right)\\
&<U_N^{(1)}\abs{A}\sum_{z\leq q<y}\frac{1}{q-1}\left(\frac{2e^\gamma}{\log\frac{D}{q}}+\frac{298.87013}{\log^\frac{3}{2}\frac{D}{q}}+\frac{1}{5\log^2 z}\right)\\
&\qquad +\frac{1.2e^{-8}U_N^{(1)}N}{\log^3 N}\left(\frac{2e^\gamma}{\log\frac{D}{y}}+\frac{298.87013}{\log^\frac{3}{2}\frac{D}{y}}+\frac{1}{5\log^2 z}\right)\\
&<U_N^{(1)}\abs{A}\left(2e^\gamma\sum_{z\leq q<y}\frac{1}{(q-1)\log\frac{D}{q}}+298.87013\sum_{z\leq q<y}\frac{1}{(q-1)\log^\frac{3}{2}\frac{D}{q}}+\frac{1}{5\log^2 z}\sum_{z\leq q<y}\frac{1}{q-1}\right)\\
&\qquad +\frac{e^{-4}U_N^{(1)}N}{\log^4 N}.
\end{split}
\end{equation}

We use Lemma \ref{lm21}, with the aid of (\ref{eq23}), to obtain
\begin{equation}
\begin{split}
\sum_{z\leq q<y}\frac{1}{q\log\frac{N^\frac{1}{2}}{q}}
<& \int_z^y \frac{dt}{t\log t\log\frac{N^\frac{1}{2}}{t}}+\frac{6}{\log N}\times\frac{1.001}{5\log^2 z} \\
<& \frac{2\log 6}{\log N}+\frac{76.9}{\log^3 N}
\end{split}
\end{equation}
and
\begin{equation}
\begin{split}
\sum_{z\leq q<y}\frac{1}{q\log^\frac{3}{2}\frac{N^\frac{1}{2}}{q}}
<& \int_z^y \frac{dt}{t\log t\log^\frac{3}{2}\frac{N^\frac{1}{2}}{t}}+\frac{6^\frac{3}{2}}{\log^\frac{3}{2} N}\times \frac{1.00001}{5\log^2 z}\\
<& \frac{4\sqrt{2}\left(\sqrt{3}-\sqrt{\frac{4}{3}}\right)}{\log^\frac{3}{2} N}+\frac{188.2}{\log^\frac{7}{2} N}.
\end{split}
\end{equation}
These inequalities, combined with $\log D>\frac{1}{2}\log N-12\log\log N-50>\frac{1}{2}\log N-14\log\log N$, give the upper bounds
\begin{equation}\label{eq56}
\begin{split}
\sum_{z\leq q<y}\frac{1}{(q-1)\log\frac{D}{q}}<& \frac{1}{1-\frac{29\log\log N}{\log N}}\left(\frac{2\log 6}{\log N}+\frac{76.9}{\log^3 N}\right)\\
<& \frac{2\log 6}{\log N}+\frac{215.02\log\log N}{\log^2 N}+\frac{76.91}{\log^3 N}
\end{split}
\end{equation}
and
\begin{equation}\label{eq57}
\begin{split}
\sum_{z\leq q<y}\frac{1}{q\log^\frac{3}{2}\frac{D}{q}}<& \frac{1}{1-\frac{29\log\log N}{\log N}}\left(\frac{4\sqrt{2}\left(\sqrt{3}-\sqrt{\frac{4}{3}}\right)}{\log^\frac{3}{2} N}+\frac{188.2}{\log^\frac{7}{2} N}\right)\\
<& \frac{3.26599}{\log^\frac{3}{2} N}+\frac{188.21}{\log^\frac{7}{2} N}.
\end{split}
\end{equation}
Moreover, (\ref{eq23}) immediately gives
\begin{equation}\label{eq58}
\sum_{z\leq q<y}\frac{1}{q-1}<\sum_{q\geq z}\frac{1}{q(q-1)}+\sum_{z\leq q<y}\frac{1}{q}<\log\frac{8}{3}+\frac{1}{3\log^2 N}.
\end{equation}

Combining (\ref{eq56})-(\ref{eq58}) with (\ref{eq55}), we obtain
\begin{equation}
\begin{split}
\sum_{z\leq q<y}S(A_q, P^{(1)}(z))<& \frac{U_N^{(1)}\abs{A}}{\log N}\\
&\times \left(4e^\gamma\log 6+\frac{976.1256}{\log^\frac{1}{2} N}\right)\\
\end{split}
\end{equation}
Thus, with the aid of (\ref{eq40b}), we prove Theorem \ref{th5}.

\section{An upper bound for a bilinear form}

We shall finish our sieve argument by obtaining an upper bound for $S(B, P(y))$:
\begin{thm}\label{th6}
Let $\ep$ be a positive real number $<0.01$.  Then
\begin{equation}
\begin{split}
& S(B, P(y))\\
\leq & 
\frac{c_2(1+\ep)NU_N}{\log^2 N}
\left(4e^\gamma(1+\ep_0(N))+\frac{860.16295}{\log^\frac{3}{2} N}\right)+\frac{e^{-138}N}{\ep \log^3 N}.
\end{split}
\end{equation}
\end{thm}

In order to obtain this upper bound for $S(B, P(y))$, we use 
the following upper bound for a bilinear form:

\begin{lem}\label{lm61}
Let $a(n)$ be an arithmetic function with $\abs{a(n)}\leq 1$ for all $n$.
Let $X, Y, Z>0$ be real numbers with $\log^{10} Y<X\leq \frac{N}{z}, XY<2N$ and $Y, Z>z$.
Moreover, let $D^*=\frac{(XY)^\frac{1}{2}}{\log^{10} Y}$.
Then we have
\begin{equation}
\begin{split}
&\sum_{d<D^*, q_1\nmid d}\max_{(a, d)=1}\abs{\sum_{n<X}\sum_{\substack{Z\leq p<Y,\\ np\equiv a\pmod{d}}}a(n)-\frac{1}{\vph(d)}\sum_{n<X}\sum_{\substack{Z\leq p<Y,\\ (np, d)=1}}a(n)}\\
&\quad \leq\frac{e^{-144}XY}{\log^4 Y}.
\end{split}
\end{equation}
\end{lem}
\begin{proof}
Factoring $\chi=\chi_{0, s}\chi_1$ with $d=rs$, the bilinear form is
\begin{equation}
\begin{split}
&\sum_{\substack{rs<D^*,\\ q_1\nmid rs}}\sideset{}{^*}\sum_{\chi\pmod{r}}\abs{\sum_{\substack{n<X,\\ (n, s)=1}}a(n)\chi(n)}\abs{\sum_{Z\leq p<Y, p\nmid s}\chi(p)}\\
\leq &\sum_{\substack{s<D^*,\\ q_1\nmid s}}\sum_{\substack{r<D^*,\\ q_1\nmid r}}\sideset{}{^*}\sum_{\chi\pmod{r}}\abs{\sum_{\substack{n<X,\\ (n, s)=1}}a(n)\chi(n)}\abs{\sum_{Z\leq p<Y, p\nmid s}\chi(p)}.
\end{split}
\end{equation}

We begin by estimating the inner sum restricted to $r<D_0$, where $D_0=\log^{10} Y$.
Since $r$ is nonexceptional, as we derived Corollary \ref{cor231} from Lemma \ref{lm23},
we derive from Theorem 1.1 of \cite{Ymd1}
that $\sideset{}{^*}\sum_{\chi\pmod{r}}\abs{\pi(x; \chi)}<\frac{10^{-4}x}{\log^{10} x}<\frac{e^{-212}x}{\log^4 x}$ for $x\geq z\geq x_0^{1/8}$.
Thus we have
\begin{equation}
\begin{split}
&\sideset{}{^*}\sum_{\chi\pmod{r}}\abs{\sum_{Z\leq p<Y, p\nmid s}\chi(p)}\\
\leq &\sideset{}{^*}\sum_{\chi\pmod{r}}\abs{\pi(Y; \chi)-\pi(Z; \chi)}+\omega(s)\\
<&\frac{e^{-212}Y}{\log^4 Y}+\frac{e^{-212}Z}{\log^4 Z}+\frac{1.3841\vph(r)\log D^*}{2\log\log D^*}\\
<&\frac{2e^{-212}Y}{\log^4 Y}+\frac{1.3841\vph(r)\log D^*}{2\log\log D^*}.
\end{split}
\end{equation}
Now we can see that the inner sum restricted to $r<D_0$ is
\begin{equation}\label{eq61}
\begin{split}
<&\frac{2e^{-212}Y}{\log^4 Y}\sum_{r<D_0}\frac{1}{\vph(r)}+\frac{1.3841D_0\log D^*}{2\log\log D^*}\\
<&\frac{4e^{-212}Y\log D_0}{\log^4 Y}+\frac{1.3841D_0\log D^*}{2\log\log D^*}.\\
<&\frac{e^{-210}Y\log D_0}{\log^4 Y},
\end{split}
\end{equation}
where the last inequality follows from fact that $D_0<\log^{10} Y$ and $\log D^*<\log X<\log N=8\log z<8\log Y$.

In order to estimate the inner sum restricted to $D_0\leq r<D^*$,
we divide the interval into intervals of the form $D_1\leq r<2D_1$, where $D_1=2^kD_0 (0\leq k\leq \frac{\log(D^*/D_0)}{\log 2})$
and use Cauchy's inequality and the large-sieve inequality.
We have, for each $D_1$,
\begin{equation}
\begin{split}
&\sum_{\substack{D_0\leq r<D^*,\\ D_1\leq r<2D_1,\\ q_1\nmid r}}\sideset{}{^*}\sum_{\chi\pmod{r}}\abs{\sum_{\substack{n<X,\\ (n, s)=1}}a(n)\chi(n)}\abs{\sum_{Z\leq p<Y, p\nmid s}\chi(p)}\\
\leq & \frac{1}{D_2}\left(\sum_r \sideset{}{^*}\sum_{\chi\pmod{r}}\frac{r}{\vph(r)}\abs{\sum_{\substack{n<X,\\ (n, s)=1}}a(n)\chi(n)}^2\right)^\frac{1}{2}\\
& \times \left(\sum_r \sideset{}{^*}\sum_{\chi\pmod{r}}\frac{r}{\vph(r)}\abs{\sum_{Z\leq p<Y, p\nmid s}\chi(p)}^2\right)^\frac{1}{2} \\
\leq & \frac{1}{D_2}\left((D_2^2+X)(D_2^2+Y)XY\right)^\frac{1}{2}\\
\leq & \sqrt{XY}(D_2+2\sqrt{X+Y})+\frac{XY}{D_2}\\
\leq & D_2\sqrt{XY}+\frac{4XY}{\log^{10} Y}+\frac{XY}{D_2},
\end{split}
\end{equation}
where $D_2$ is the number of integers $r$ with $\max\{D_0, D_1\}\leq r<\min\{2D_1, D^*\}$,
and summming these quantities over $0\leq k\leq \frac{\log(D^*/D_0)}{\log 2}$, we obtain
\begin{equation}
\begin{split}
&\sum_{\substack{D_0\leq r<D^*,\\ q_1\nmid r}}\sideset{}{^*}\sum_{\chi\pmod{r}}\abs{\sum_{\substack{n<X,\\ (n, s)=1}}a(n)\chi(n)}\abs{\sum_{Z\leq p<Y, p\nmid s}\chi(p)}\\
<& 2D^* \sqrt{XY}+\left(\frac{4\log D^*}{\log 2}+2\right)\frac{XY}{D_0}\\
<& \frac{4XY}{\log^{10} Y}+\frac{2XY\log N}{(\log 2)\log^{10} Y}<\frac{2.9XY\log N}{\log^{10} Y}.
\end{split}
\end{equation}
Summing over $s$ is only to multiply by $\sum_s\frac{1}{\vph(s)}<2\log D^*<\log N=8\log z<8\log Y$,
so that the contribution is at most $\frac{2e^{-145}XY}{\log^4 Y}$.  Combining this estimate
with (\ref{eq61}), we obtain the result.
\end{proof}

The rest of this section is devoted to the proof of Theorem \ref{th6}.
Let
\begin{equation}
B^{(j)}=\{N-p_1p_2p_3: z\leq p_1<y\leq p_2\leq p_3, up_2p_3<N, (p_2p_3, N)=1, w_j\leq p_1<(1+\ep)w_j\},
\end{equation}
where
\begin{equation}
w_j=z(1+\ep)^j\textrm{ for }j=0, 1, \cdots, j_0=\ceil{\frac{\log(y/z)}{\log(1+\ep)}}-1,
\end{equation}
and
$B^\#=\cup_j B^{(j)}$.

We can easily see that
$\abs{B^{(j)}}=(\pi(Y)-\pi(Z))\#\{(p_2, p_3): y\leq p_2\leq p_3, w_j p_2p_3<N, (p_2p_3, N)=1\}$,
where $Z=\max\{z, w_j\}$ and $Y=\min\{(1+\ep)w_j, y\}$,
and
\begin{equation}\label{eq62}
S(B, P^{(1)}(y))\leq S(B^\#, P^{(1)}(y))\leq\sum_{j=0}^{j_0} S(B^{(j)}, P^{(1)}(y)).
\end{equation}

\begin{lem}\label{lm6x1}
\begin{equation}
\abs{B^\#}<\frac{c_2(1+\ep)N}{\log N}+\frac{2.207(1+\ep)N}{\log^2 N},
\end{equation}
where $c_2=\int_{1/8}^{1/3}\frac{\log(2-3\beta)}{\beta(1-\beta)}+10^{-8}<0.36309$.
\end{lem}
\begin{proof}
We begin by
\begin{equation}\label{eq6x1}
\begin{split}
\abs{B^\#}&\leq\sum_{\substack{z\leq p_1<y\leq p_2,\\ p_1p_2^2<(1+\ep)N}}\pi\left(\frac{(1+\ep)N}{p_1p_2}\right)\\
&\leq (1+\ep+10^{-8})N\sum_{z\leq p_1<y}\frac{1}{p_1}\sum_{y\leq p_2<w}\frac{1}{p_2\log\frac{N}{p_1p_2}},
\end{split}
\end{equation}
where $w=w(p_1)=\sq{\frac{(1+\ep)N}{p_1}}$.

As in p. 289 in \cite{Nat}, we introduce two functions
$h_p(t)=\frac{1}{\log\frac{N}{pt}}$ and
\begin{equation}
H(u)=\int_y^{(N/u)^\frac{1}{2}}h_u(t)d\log\log t=\int_y^{(N/u)^\frac{1}{2}}\frac{d\log\log t}{\log\frac{N}{ut}}.
\end{equation}
We see that
\begin{equation}
H(N^\al)=\int_{N^\frac{1}{3}}^{N^{\frac{1-\al}{2}}}\frac{d\log\log t}{\log\frac{N^{1-\al}}{t}}
=\frac{1}{\log N}\int_{\frac{1}{3}}^{\frac{1-\al}{2}}\frac{d\beta}{\beta(1-\al-\beta)}.
\end{equation}
In particular, we have $h_{p_1}(w)=\frac{2}{\log\frac{N}{(1+\ep)p_1}}$, $H(y)=0$ and $H(z)=\frac{\log\frac{26}{21}}{\log N}$.

Now Lemma \ref{lm21} and (\ref{eq23}) give that the inner sum in (\ref{eq6x1}) is 
\begin{equation}
\sum_{y\leq p_2<x}\sum\frac{h_{p_1}(p_2)}{p_2}\leq H(p_1)+\frac{h_{p_1}(w)}{4\log^2 y}.
\end{equation}
Since 
\begin{equation}
\begin{split}
\frac{h_{p_1}(w)}{4\log^2 y}=& \frac{1}{2\log\frac{N}{(1+\ep)p_1}\log y} \\
\leq & \frac{1}{2\log\frac{N}{(1+\ep)y}\log y} \\
= & \frac{1}{2\log\frac{N^\frac{2}{3}}{1+\ep}\log N^\frac{1}{3}} \\
\leq & \frac{9}{4\log^2 N}\times \frac{\frac{2}{3}\log N}{\frac{2}{3}\log N-\log(1+\ep)} \\
\leq & \frac{9(1+e^{-50})}{4\log^2 N},
\end{split}
\end{equation}
we have
\begin{equation}
\sum_{y\leq p_2<x}\sum\frac{h_{p_1}(p_2)}{p_2}\leq H(p_1)+\frac{9(1+e^{-50})}{4\log^2 N}.
\end{equation}

Hence the outer sum in (\ref{eq6x1}) is at most
\begin{equation}
\sum_{z\leq p_1<y}\frac{H(p_1)}{p_1}+\frac{9(1+e^{-50})}{4\log^2 N}\sum_{z\leq p_1<y}\frac{1}{p_1}
<\sum_{z\leq p_1<y}\frac{H(p_1)}{p_1}+\frac{2.2069}{\log^2 N}
\end{equation}
since it immediately follows from (\ref{eq23}) that $\sum_{z\leq p<y}(1/p)<\log(8/3)+\frac{1}{4\log^2 z}$.

Using Lemma \ref{lm21} and (\ref{eq23}) again and exploiting the fact
$\int_z^y H(u)d\log\log u=\frac{c_2}{\log N}$ in p. 291 of \cite{Nat}, we have
\begin{equation}
\begin{split}
\sum_{z\leq p_1<y}\frac{H(p_1)}{p_1}<& \int_z^y H(u)d\log\log u+\frac{\log\frac{26}{21}}{4\log^2 z\log N}\\
=& \frac{c_2}{\log N}+\frac{16\log\frac{26}{21}}{\log^3 N}.
\end{split}
\end{equation}
This proves the lemma.
\end{proof}

We set $D=\frac{N^\frac{1}{2}}{\log^{10} N}$ and
let $s=\frac{\log D}{\log y}$.  Then Theorem \ref{th31} gives
\begin{equation}
\begin{split}
& S(B^{(j)}, P^{(1)}(y))\\
\leq & \abs{B^{(j)}}\left(V^{(1)}(y)+U_N^{(1)}\left(\frac{F(s)}{\log D}+\frac{298.87013}{\log^\frac{3}{2}D}\right)\right)+\sum_{d<D, d\mid P^{(1)}(y)}\abs{r^{(j)}_d},
\end{split}
\end{equation}
where $r^{(j)}_d=\abs{B^{(j)}_d}-\frac{\abs{B^{(j)}}}{\vph(d)}$.
By Lemma \ref{lm22}, we have $V^{(1)}(y)<\frac{U_N^{(1)}}{\log y}\left(1+\frac{1}{5\log^2 y}\right)$.
We can see that $1<s_w<3$ and therefore $F(s_l)=2e^\gamma-s_w$.
Thus, similarly to (\ref{eq53}), we have, for each $j$,
\begin{equation}
\begin{split}
& S(B^{(j)}, P^{(1)}(y))\\
\leq & \abs{B^{(j)}}U_N^{(1)}\left(\frac{2e^\gamma}{\log D}+\frac{298.87013}{\log^\frac{3}{2}D}+\frac{1}{5\log^2 y}\right)+R^{(j)}
\end{split}
\end{equation}
and therefore, by (\ref{eq62}),
\begin{equation}\label{eq63}
\begin{split}
& S(B, P^{(1)}(y))\\
\leq & \abs{B^\#}U_N^{(1)}\left(\frac{2e^\gamma}{\log D}+\frac{298.87013}{\log^\frac{3}{2}D}+\frac{1}{5\log^2 y}\right)+R,
\end{split}
\end{equation}
where $R^{(j)}=\sum_{d<D, d\mid P^{(1)}(y)}\abs{r^{(j)}_d}$ and $R=\sum_{j=0}^{j_0} R^{(j)}$.

We put $a(n)=a_N(n)$ to be the characteristic function of the set of integers of the form $n=p_2p_3$
with $y\leq p_2<p_3$ and $(N, p_2p_3)=1$.  Then, noting that $(d, p_2p_3)=1$ since $d\mid P(y)$, we see that
\begin{equation}
\begin{split}
r^{(j)}_d=& \sum_{n<N/w_j}\sum_{\substack{Z\leq p<Y,\\ np\equiv N\pmod{d}}}a(n)-\frac{1}{\vph(d)}\sum_{n<N/w_j}\sum_{Z\leq p<Y}a(n)\\
=& r^{(j, 1)}_d+r^{(j, 2)}_d,
\end{split}
\end{equation}
where
\begin{equation}
r^{(j, 1)}_d=\sum_{n<N/w_j}\sum_{\substack{Z\leq p<Y,\\ np\equiv N\pmod{d}}}a(n)-\frac{1}{\vph(d)}\sum_{n<N/w_j}\sum_{\substack{Z\leq p<Y,\\ (np, d)=1}}a(n)
\end{equation}
and
\begin{equation}
r^{(j, 2)}_d=\frac{1}{\vph(d)}\sum_{n<N/w_j}\sum_{\substack{Z\leq p<Y,\\ p\mid d}}a(n).
\end{equation}
Now we can divide $R\leq R_1+R_2$,
where $R_i=\sum_{j=0}^{j_0}\sum_{d<D, d\mid P^{(1)}(y)}\abs{r^{(j, i)}_d}$ for $i=1, 2$.

In order to estimate $R_1$, we shall apply Lemma \ref{lm61} with $X=N/w_j, Y=\min\{y, (1+\ep)_j\}, Z=\max\{w_j, z\}$ and $a=N$.
We see that $D<D^*=\frac{(XY)^\frac{1}{2}}{\log^{10} Y}<(XY)^\frac{1}{2}<N$ and therefore
\begin{equation}
\begin{split}
& \sum_{d<D, d\mid P(y)}\abs{r^{(j, 1)}_d}\\
\leq & \sum_{d<D^*, d\mid P(y)}\abs{r^{(j, 1)}_d}\\
= & \sum_{d<D^*, d\mid P(y)}\abs{\sum_{n<N/w_j}\sum_{\substack{Z\leq p<Y,\\ np\equiv N\pmod{d}}}a(n)-\frac{1}{\vph(d)}\sum_{n<N/w_j}\sum_{\substack{Z\leq p<Y,\\ (np, d)=1}}a(n)}\\
< & \frac{e^{-144}(1+\ep)N}{\log^4 y}
\end{split}
\end{equation}
for each $j=0, 1, \ldots, j_0$.
Hence, noting that $\ep<1/100$, we see that $R_1$ is at most
\begin{equation}\label{eq64}
\begin{split}
\sum_{j=0}^{j_0}\sum_{d<D, d\mid P(y)}\abs{r^{(j, 1)}_d}&<\frac{(1+\ep)e^{-144}N\log\frac{y}{z}}{\log(1+\ep) \log^4 y}\\
&<\frac{e^{-139}N}{\ep \log^3 N}.
\end{split}
\end{equation}

Next, $R_2$ is at most
\begin{equation}\label{eq65}
\begin{split}
\sum_{j=0}^{j_0}\sum_{d<D, d\mid P(y)} r^{(j, 2)}_d=& \frac{1}{\vph(d)}\sum_{n<N/l}\sum_{\substack{Z\leq p<Y,\\ p\mid d}}a(n)\\
\leq& 8N^\frac{7}{8}\sum_{d<D, d\mid P(y)}\frac{1}{\vph(d)}\\
<& 8N^\frac{7}{8}\log N\\
<& \frac{e^{-139}N}{\log^3 N}
\end{split}
\end{equation}
since we can see that
\begin{equation}
\begin{split}
\sum_{j=0}^{j_0} r^{(j, 2)}_d=& \frac{1}{\vph(d)}\sum_{n<N/w_j}\sum_{\substack{Z\leq p<Y,\\ p\mid d}}a(n)\\
=& \frac{1}{\vph(d)}\sum_{n<N/w_j}\sum_{p>z, p\mid d}a(n)\\
\leq& \frac{N\log d}{\vph(d)l\log z}\leq\frac{8N^\frac{7}{8}}{\vph(d)}.
\end{split}
\end{equation}

Using (\ref{eq64}) and (\ref{eq65}), we have
\begin{equation}\label{eq66}
\begin{split}
& S(B, P^{(1)}(y))\\
\leq & \abs{B^\#}U_N^{(1)}\left(\frac{2e^\gamma}{\log D}+\frac{298.87013}{\log^\frac{3}{2}D}+\frac{1}{5\log^2 y}\right)+\frac{2e^{-139}N}{\ep \log^3 N}.
\end{split}
\end{equation}

Now Lemma \ref{lm6x1} gives
$\abs{B^\#}<\frac{c_2(1+\ep)N}{\log N}+\frac{2.207(1+\ep)N}{\log^2 N}$.
Since $\log D=\frac{1}{2}\log N-10\log\log N$, we have
$1/\log D-2/\log N<20 \log\log N/(\log N-20\log\log N)<e^{-10}/\log^\frac{3}{2} N$
and therefore the right-hand side in (\ref{eq66}) is at most
\begin{equation}
\frac{c_2(1+\ep)NU_N^{(1)}}{\log N}\left(\frac{4e^\gamma}{\log N}+\frac{845.33255}{\log^\frac{3}{2} N}\right)+\frac{2e^{-139}N}{\ep \log^3 N}.
\end{equation}

We recall the inequality (\ref{eq40b}) and obtain
\begin{equation}
\begin{split}
& S(B, P^{(1)}(y))\\
\leq & 
\frac{c_2(1+\ep)NU_N}{\log^2 N}
\left(4e^\gamma(1+\ep_0(N))+\frac{860.16295}{\log^\frac{3}{2} N}\right)+\frac{2e^{-139}N}{\ep \log^3 N}.
\end{split}
\end{equation}
Since trivially $S(B, P^{(1)}(y))\leq S(B, P(y))$, we obtain Theorem \ref{th6}.

\section{Proof of the main theorem}
We recall that
\begin{equation}
V^{(j)}(x)=\frac{U^{(j)}_N}{\log x}\left(1+\frac{\theta}{5\log^2 x}\right)\left(1+\frac{8\theta\log x}{x}\right)
\end{equation}
for $j=0, 1, \ldots, l$ and $x\geq z$, where
\begin{equation}
U_N^{(j)}=2e^{-\gamma}\prod_p\left(1-\frac{1}{(p-1)^2}\right)\prod_{p>2, p\mid Nm_j}\frac{p-1}{p-2}.
\end{equation}

Now we shall take $e^{-100}<\ep<e^{-20}$ and apply Lemma \ref{lm41} combined with
Theorems \ref{th4}, \ref{th5} and \ref{th6}, which gives
\begin{equation}
\begin{split}
\pi_2(N)\frac{\log N}{U_N\abs{A}}>
&e^\gamma(4\log 3-2\log 6-2c_2(1+\ep))\\
& -\ep_0(N)(2e^\gamma(c_2(1+\ep)+\log 6)+0.5198)\\
& -\frac{767.7471+496.6254+430.0815c_2(1+\ep)}{\log^\frac{1}{2} N}-\frac{1}{\log N}.
\end{split}
\end{equation}

Since $\ep_0(N)\leq 1/57$ and $\ep<e^{-20}$, we have
\begin{equation}
\pi_2(N)\frac{\log N}{U_N\abs{A}}>0.007.
\end{equation}
As mentioned in (\ref{eq41a}), we have $\abs{A}>\frac{N}{\log N}$.
This completes the proof of the main theorem.

{}

{\small Tomohiro Yamada}\\
{\small Center for Japanese language and culture\\Osaka University\\562-8558\\8-1-1, Aomatanihigashi, Minoo, Osaka\\Japan}\\
{\small e-mail: \protect\normalfont\ttfamily{tyamada1093@gmail.com}}
\end{document}